\documentclass{article}[12pt]

\usepackage[all]{xy}
\usepackage{mathtools}
\usepackage{wrapfig}
\usepackage{pgf,tikz}
\usetikzlibrary{arrows}
\usepackage{hyperref}
\usepackage{listings} 
\usepackage[utf8]{inputenc}
\usepackage{amsmath,amsthm,amsfonts,amssymb,fancyhdr,titlesec,graphics,enumerate,color,xfrac,tikz,yfonts,float,mathabx, authblk}
  
  \AtEndDocument{%
  \par
  \medskip
  \begin{tabular}{@{}l@{}}%
    \textsc{Department of Mathematics} \\
     \textsc{University of the Basque Country UPV/EHU}  \\ \textsc{48080 Bilbao, Spain}\\
    \textit{E-mail address}: \texttt{jon.gonzalez@ehu.eus}
  \end{tabular}}
  
    \AtEndDocument{%
  \par
  \medskip
  \begin{tabular}{@{}l@{}}%
    \textsc{Department of Mathematics} \\
     \textsc{University of the Basque Country UPV/EHU}  \\ \textsc{48080 Bilbao, Spain}\\
    \textit{E-mail address}: \texttt{andoni.zozaya@ehu.eus}
  \end{tabular}}

\newcommand\blfootnote[1]{%
  \begingroup
  \renewcommand\thefootnote{}\footnote{#1}%
  \addtocounter{footnote}{-1}%
  \endgroup
}

\theoremstyle{plain}

\newtheorem{lemma}{Lemma}
\newtheorem{proposition}{Proposition}
\newtheorem{theorem}{Theorem}
\newtheorem{corollary}{Corollary}

\newtheorem{conjecture}{Conjecture}

\theoremstyle{remark}

\newtheorem{remark}[lemma]{Remark}
\newtheorem*{remark*}{Remark}

\newtheorem*{notation*}{Notation}
\newtheorem{question}{Question}
\newtheorem*{question*}{Question}

\theoremstyle{definition}

\newtheorem*{definition*}{Definition}
\newtheorem*{remarks*}{Remarks}
\newtheorem*{example*}{Example}
\newtheorem*{examples*}{Examples}
\newtheorem*{conjecture*}{Conjecture}
\newtheorem*{conjectures*}{Conjectures}

\def\N{\mathbb{N}}
\def\Z{\mathbb{Z}}
\def\G{\mathcal{G}}
\def\F{\mathbb{F}}
\def\B{\mathcal{B}}

\DeclareMathOperator{\mideal}{\mathfrak{m}}
\DeclareMathOperator{\Hcal}{\mathcal{H}}
\DeclareMathOperator{\Sp}{Sp}
\DeclareMathOperator{\SO}{SO}
\DeclareMathOperator{\GL}{GL}
\DeclareMathOperator{\SL}{SL}

\DeclareMathOperator{\hspec}{hspec}
\DeclareMathOperator{\sthspec}{hspec_{st}}
\DeclareMathOperator{\hdim}{hdim}
\DeclareMathOperator{\Mat}{M}

\DeclareMathOperator{\hD}{hD}
\DeclareMathOperator{\sthdim}{hdim_{st}}
\DeclareMathOperator{\grm}{gr }
\DeclareMathOperator{\pr}{pr}
\DeclareMathOperator{\Hsthdim}{hdim_{st}^{\textit{H}}}

\title{Standard Hausdorff spectrum of compact $\F_p[[t]]$-analytic groups}
\author{Jon Gonz\'alez-S\'anchez and Andoni Zozaya}
\date{}

\begin{document}

\maketitle

\begin{abstract}
\blfootnote{
The authors are supported by the Basque Government grant IT974-16 and by the Spanish Government grant MTM2017-86802-P, partly with ERDF funds. The second author is supported as well by the Spanish Ministry of Science, Innovation and Universities under grant FPU17/04822. \newline
 }We prove that the $\F_p[[t]]$-standard Hausdorff spectrum of a compact $\F_p[[t]]$-analytic group contains a real interval and that it coincides with the full unit interval when the group is soluble. Moreover, we show that the $\F_p[[t]]$-standard Hausdorff spectrum of classical Chevalley groups over $\F_p[[t]]$ is not full, since $1$ is an isolated point thereof.
\end{abstract}

\section{Introduction}

The concept of Hausdorff dimension arose as a generalisation of the notion of topological dimension. This dimension can be defined in any metric space; and in the specific group theoretical context, the study of the Hausdorff dimension in profinite groups has attracted much attention. \\

 If $G$ is a countably based profinite infinite group, a \emph{filtration series} of $G$ is a family  $\{ G_n\}_{n \in \N}$ of descending open subgroups which is a neighbourhood system of the identity, that is, $\bigcap_{n \in \N} G_n = \{ 1 \}$. Such a filtration defines a metric on $G$ by letting
$$d(x, y) = \inf \left\{ |G : G_n|^{-1} \mid xy^{-1} \in G_n \right\}.$$
This notion of distance makes $G$ a metric space and so one can define the Hausdorff dimension of a subset $X \subseteq G$ with respect to that filtration (cf. \cite[Section 2]{BSh} and \cite[Chapter 3]{Falconer}); it will be denoted by $\hdim_{\{ G_n \}}(X)$ or $\hdim(X)$ if there is no risk of confusion. Further, when a filtration consists of normal subgroups it is called \emph{normal filtration}. It was proved in \cite[Theorem 2.4]{BSh} that when the filtration is normal and $H$ is a closed subgroup of $G$ then  one can compute the Hausdorff dimension by the following formula:
\begin{equation}
\label{hdim}
\hdim_{\{G_n\}}(H) = \liminf_{n \rightarrow \infty} \frac{\log |HG_n: G_n |}{ \log|G:G_n|}.
\end{equation}

It has been repeatedly pointed out that the Hausdorff dimension may depend on the chosen filtration. Furthermore,  for a fixed filtration $\{G_n\}_{n \in \N}$ we can consider the collection of all the values $\hdim_{\{ G_n \}}(H)$ as $H$ ranges over closed subgroups of $G,$ that is, the set
$$\hspec_{\{ G_n \}}(G) = \left\{ \hdim_{\{ G_n \}}(H) \mid H \leq_c G \right\},$$
which is called the \emph{Hausdorff spectrum} of $G$ with respect to the filtration series $\{ G_n \}_{n \in \N}.$
It turns out that these families may have little or no resemblance as one changes the filtration.  For example, consider the additive $p$-adic analytic group $\Z_p \oplus \Z_p.$ For finitely generated pro-$p$ groups of this kind there exists a natural filtration series, namely the \emph{$p$-power filtration} given by $G_n = G^{p^n}.$ It is immediate to see that with respect to this filtration series one has  $\hspec_{\{ G_n \}}(\Z_p \oplus \Z_p) = \{ 0, 1/2, 1\},$ and so  in particular it is finite. 

However, in \cite[Theorem 1.3]{KTZ} it is shown that there exists a filtration series $\{ G_n \}_{n \in \N}$ such that $\hspec_{\{ G_n \}}(\Z_p \oplus \Z_p)$ contains the real interval $\left[\frac{1}{p+1}, \frac{p-1}{p+1} \right].$ Thus, even the finiteness of the Hausdorff spectrum is not filtration invariant. 

According to \cite[Corollary 1.2]{BSh} $\hspec_{\{G^{p^n}\}}(G)$ is finite for any $p$-adic analytic pro-$p$ group $G,$ which suggests the following classical question (cf. \cite[Problem 1]{BSh}):

\begin{question}
\label{BS}
Let $G$ be a finitely generated pro-$p$ group such that $\hspec_{\left\{G^{p^n}\right\}}(G)$ is finite. Is $G$ $p$-adic analytic?
\end{question}

Clearly, although the conjecture is stated here for the $p$-power filtration, it can also be posed for many other different non-pathological filtrations (some results in this direction can be found in \cite{KTZ}). 

We will work in the setting of $R$-analytic groups where $R$ is a pro-$p$ domain; these comprise an abstract group together with an $R$-analytic manifold structure in such a way that both structures are compatible in the sense that the multiplication map and the inversion map are $R$-analytic functions. They are thoroughly studied in \cite{DDMS} and \cite{Serre}. 

It can be proved that an $R$-analytic group is profinite if and only if it is compact, and thus formula (\ref{hdim}) (with respect to any normal filtration) holds for compact $R$-analytic groups. 

In this family of groups the $p$-power filtration series can not be used in general. Indeed, $G^{p^n}$ will normally not be an open subgroup of a compact $R$-analytic group $G.$ However, they possess  a canonical filtration series, called the $R$-standard filtration series, which depends only on the $R$-analytic manifold structure of $G.$ The Hausdorff dimension relative to this filtration series -- which is introduced insightfully in Section \ref{3}-- is called the $R$-standard Hausdorff dimension.

In the present paper, we shall mostly restrict to the case $R = \F_p[[t]],$ and the main findings of this investigation are:

\begin{theorem}
\label{T1}
If $G$ is a soluble compact $\F_p[[t]]$-analytic group then the Hausdorff spectrum of $G$ with respect to the $\F_p[[t]]$-standard filtration is $[0,1].$
\end{theorem}

\begin{theorem}
\label{T2}
If $G$ is a compact $\F_p[[t]]$-analytic group then the Hausdorff spectrum of $G$ with respect to the $\F_p[[t]]$-standard filtration contains the real interval $[0,\alpha]$ for some $\alpha \geq 1 /  \dim{G}.$
\end{theorem}

In the latter result, $\dim{G}$  denotes the analytic dimension of $G$  as an $\F_p[[t]]$-analytic manifold. In addition, in Corollary \ref{alphades}  the $\alpha$ occurring in the statement of Theorem \ref{T2} is described more accurately for classical Chevalley groups over $\F_p[[t]]$; in particular we shall show that they always satisfy $\alpha \geq 1/2.$ Furthermore, in Corollary \ref{batiso} we shall prove that for most of these groups 1 is an isolated point in the spectrum, providing some examples of compact $\F_p[[t]]$-analytic groups whose spectrum with respect to the $\F_p[[t]]$-standard filtration series is not full.\\

Finally,  we outline a consequence which can be derived from Theorems \ref{T1} and \ref{T2}. An $R$-analytic subgroup is a structure which occurs both as a subgroup and a submanifold (for the latter we adopt Serre's definition in \cite[Part II, Section III.11]{Serre});  for example any open subgroup is an $R$-analytic subgroup of maximal dimension. According to \cite[Main Theorem]{FGG}, the $R$-standard Hausdorff dimension of an $R$-analytic subgroup can only take finitely many rational values. However, it follows from Theorem \ref{T2} that the $\F_p[[t]]$-standard spectrum of compact $\F_p[[t]]$-analytic groups is uncountable; hence showing that there are numerous closed subgroups that are not $\F_p[[t]]$-analytic.\\
\linebreak
\textit{Notation}
Most of the notation is standard except $X^{(n)},$ which denotes the $n$-Cartesian power of the set $X$. Apart from that,  $R$ is always a pro-$p$ domain with maximal ideal $\mathfrak{m},$ $\N$ is the set of natural numbers (including $0$), $p$ is a prime number, $\F_p$ is the finite field of $p$ elements, $\Z_p$ is the ring of $p$-adic integers and $R[[X]]$ is the power series ring with coefficients in $R$.  Moreover, and $H \leq_o G$ (resp. $H \leq_c G$) means that $H$ is an open (resp. closed) subgroup of a topological group $G.$ 

\section{Preliminaries}

Throughout this article, relating the Hausdorff dimension of a countably based profinite group to that of its subgroups and quotients will be of vital importance. Therefore, it is sometimes convenient to use the notation $\hdim_{\{G_n\}}^{G}$ to emphasize that the dimension, with respect to the filtration series $\{G_n \}_{n \in \N}$, is calculated within the group $G$.  The following result is known for subgroups  (cf. \cite[Lemma 5.3]{KTZ}), and it will  be stated here for the convenience of the reader.

\begin{lemma}
\label{KTZ5.3}
Let $G$ be a countably based profinite group, $\{ G_n \}_{n \in \N}$ a normal filtration series and $H \leq_c G$ a closed subgroup whose Hausdorff dimension is given by a proper limit. Then 
$$\hdim_{\{G_n\}}^G(K) = \hdim_{\{G_n\}}^G(H) \hdim_{\{H \cap G_n \}}^H(K)$$
for all $K \leq_c H.$
\end{lemma}
 The Hausdorff dimension of $H$ above being a proper limit means that 
$$\hdim_{\{G_n \}}(H) = \lim_{n \rightarrow \infty} \frac{\log|HG_n :  G_n|}{\log{|G:G_n|}}.$$
Moreover, for quotients of countably based profinite groups we have the following result (cf. \cite[Lemma 2.2]{KT}).
\begin{lemma}
\label{zatidura}
Let $G$ be a countably based profinite group, $\{G_n\}_{n\in \N}$  a normal filtration series of $G$ and  $N \unlhd G$  a closed normal subgroup. Assume that $\hdim_{\{G_n\}}^{G}(N)$ is given by a proper limit. Then for every subgroup  $H \leq_c G$ containing $N$  one has
$$\hdim_{\{G_n\}}^{G}(H) = \left(1 - \hdim_{\{G_n\}}^G(N)\right)\hdim_{\{G_nN/N\}}^{G/N}(H/N) + \hdim_{\{G_n\}}^G(N).$$
\end{lemma}

\begin{corollary}
\label{zatidura2}
Let $G$ be a countably based profinite group with normal filtration series $\{ G_n\}_{n \in \N}$ and let $N \unlhd G$ be a finite normal subgroup. Then 
$$\hspec_{\{ G_n\} } (G) = \hspec_{\{ G_n N / N \}} (G/ N).$$
\end{corollary}
\begin{proof}
Since $\hdim_{\{G_n\}}^G(N) = 0$ is given by a proper limit, the inclusion 
$$\hspec_{\{G_nN/N\}}(G/N) \subseteq \hspec_{\{G_n\}}(G)$$
is a direct consequence of the Correspondence Theorem and Lemma \ref{zatidura}. \\

For the converse, consider $\eta \in \hspec_{\{G_n\}}(G);$ then there exists $H \leq_c G$ such that $\hdim_{\{G_n \}}^G(H) = \eta.$ Thus, since $N$ is finite and the right multiplication is an isometry by Lemma \ref{zatidura} one has
\begin{align*}
\hdim_{\{ G_n \}}^G(H)  &= 
\hdim_{\{G_n \}}^G \left( \bigcup_{n \in N} Hn\right) \\& =
 \hdim_{\{ G_n \}}^G(HN)  =  \hdim_{\{ G_nN/N \}}^G(HN/N),
\end{align*}
as required. 
 \end{proof}

Finally, the combination of the above results yields the following corollary.

\begin{corollary}
\label{korzatidura1}
Let $G$ be a countably based profinite group,  $\{G_n\}_{n\in \N}$  a normal filtration series and let $N \unlhd K \leq G$ be closed subgroups such that $\hdim_{\{G_n\}}^{G}(N) = \eta$ and $\hdim^{G}_{\{G_n \}}{(K)} = \kappa$ are given by proper limits. If $\hspec_{\left\{ \frac{(K \cap G_n) N}{ N} \right\}}(K/N) = [0,1]$ then $[\eta, \kappa] \subseteq \hspec_{\{G_n\}}(G).$ 
\end{corollary}
\begin{proof}
Firstly, by Lemma \ref{KTZ5.3} it follows that $\hdim_{\{K \cap G_n \}}^{K}(N) = \eta/ \kappa,$ and using the Correspondence Theorem and Lemma \ref{zatidura} we obtain
\begin{small}
$$[\eta/\kappa, 1] = \left\{  (1-\eta/\kappa)\alpha + \eta/\kappa \mid \alpha \in \hspec_{\left\{\frac{(K \cap G_n)N}{N} \right\}}(K/N) \right\}  \subseteq \hspec_{\{ K \cap G_n \}}(K).$$
\end{small}
By another application of Lemma \ref{KTZ5.3}, one concludes $[\eta, \kappa] \subseteq \hspec_{\{G_n\}}(G).$ \end{proof}

\section{$R$-standard Hausdorff dimension}
\label{3}

An $R$-analytic group $S$ is called \emph{$R$-standard} of level $N$ and dimension $d$ when there exist a homeomorphism $\phi \colon S \rightarrow \left(\mideal^N \right)^{(d)}$ such that $\phi(1) = \textbf{0}$, and a formal group law $\mathbf{F}$ over $R$ such that 
$$\phi(xy) = \mathbf{F}(\phi(x), \phi(y)) \textrm{ for every } x,y   \in S.$$
In that case, we usually write $(S, \phi)$ to denote the standard group, in order to emphasise the r\^{o}le of $\phi.$ Any  $R$-analytic group contains, by \cite[Theorem 13.20]{DDMS}, an open $R$-standard subgroup. In addition, by \cite[Proposition 13.22]{DDMS}, $R$-standard groups are pro-$p$ groups and so they are compact.

\begin{remark}
Let  $\mathbf{X}$ and $\mathbf{Y}$ be two $d$-tuples of indeterminates. Since the formal group law $\mathbf{F} \in R[[\mathbf{X}, \mathbf{Y}]]^{(d)}$ defines a group structure it is straightforward (cf. \cite[Proposition 13.16(i)]{DDMS}) to see that it has the form
\begin{equation}
\label{fgl}
\mathbf{F}(\mathbf{X}, \mathbf{Y}) = \mathbf{X} + \mathbf{Y} + \mathbf{G}(\mathbf{X}, \mathbf{Y}), \end{equation}
where every monomial involved in $\mathbf{G}$ has total degree at least $2$ and contains a non-zero power of $X_i$ and $Y_j$ for some $i, j \in \{1, \dots, d \}.$
\end{remark}

In the context of compact $R$-analytic groups a natural filtration is available. Indeed, let $G$ be a  compact $R$-analytic group and let $(S , \phi)$ be an open $R$-standard subgroup. An \emph{$R$-standard filtration} of $G$ (the one induced by $S$) is the filtration $\{ S_n\}_{n \in \N}$ defined by
$$S_{n} := \phi^{-1} \left( \left(\mideal^{N+n} \right)^{(d)} \right), \textrm{ } \forall n \in \N .$$

It is immediate to see that  an $R$-standard filtration is indeed a filtration. Furthermore, by \cite[Proposition 13.22]{DDMS} one has that $S_n \unlhd S$ for any $n \in \N$ and thus formula (\ref{hdim}) holds for $R$-standard groups with the above filtration. \\

Because of the dependence of $\hdim$ on the chosen filtration we should not assume \emph{a priori} that the Hausdorff dimension of a subgroup of a compact $R$-analytic group is the same when computed with respect to two different $R$-standard filtrations. However, the following result (cf. \cite[Theorem 3.1]{FGG}) shows that the $R$-standard Hausdorff dimension is independent of the standard subgroup.

\begin{theorem}
\label{standard}
Let $G$ be a compact $R$-analytic group and let $(S, \phi)$ and $(T, \psi)$ be two open $R$-standard subgroups of $G.$ Then 
$$\hdim_{\{S_n\}}(H) = \hdim_{\{ T_n\}} (H)$$
for every closed subgroup $H \leq G.$
\end{theorem}

This Hausdorff dimension, which we will denote by $\sthdim,$ is called the \emph{standard} or \emph{$R$-standard Hausdorff dimension} of $H$  and 
$$\sthspec(G) = \left\{ \sthdim(H) \mid H \leq_c G \right\}$$
is the \emph{standard} or \emph{$R$-standard Hausdorff spectrum} of $G.$ \\

Note that an $R$-analytic subgroup of a compact  $R$-analytic group $G$ is a compact $R$-analytic group in its own right, since it is a locally closed topological subgroup of a compact group; and thus its Hausdorff dimension can be computed. In particular, an $R$-standard filtration $\{ S_n \}_{n \in \N}$ defines a Hausdorff dimension in both $G$ and the open $R$-standard subgroup $S$. In the notation of the preceding section, these dimensions are denoted respectively by $\hdim_{\{S_n\}}^G$ and $\hdim_{\{S_n\}}^S.$

\begin{lemma}
Let $G$ be a compact $R$-analytic group with open $R$-standard subgroup $(S, \phi).$ Then 
$$\hdim_{\{S_n\}}^{G}(H) = \hdim_{\{ S_n \}}^{S} (H \cap S)$$
for every closed subgroup $H \leq G.$
\end{lemma}
\begin{proof}
Let $d_G$ and $d_S$ be the metrics induced by the filtration $\{ S_n \}_{n \in \N}$ in $G$ and in $S$ respectively. Then $d_G(x,y) = |G:S|^{-1}d_S(x,y)$ and so the inclusion map from $(S, d_S)$ to $(G, d_G)$ is bi-Lipschitz. Hence by \cite[Proposition 3.3]{Falconer} it follows that
$$ \hdim_{\{ S_n \}}^{S} (H \cap S) =  \hdim_{\{ S_n \}}^{G} (H \cap S).$$
Moreover, since $H \cap S$ is an open subgroup of $H$ by \cite[Lemma 2.4] {FGG} we deduce that 
$$\hdim_{\{ S_n \}}^{G} (H \cap S) = \hdim_{\{ S_n \}}^{G} (H),$$
as required. 
 \end{proof}
Thus, we have the following immediate consequence.
\begin{corollary}
\label{estandarrera}
Let $G$ be a compact $R$-analytic group with an open $R$-standard subgroup $(S, \phi).$ Then $\sthspec(G) = \sthspec(S).$
\end{corollary}

Accordingly, in order to study the standard Hausdorff spectrum of a compact $R$-analytic group we can assume that the original group $G$ is itself an $R$-standard group. 

Finally, we shall study the standard Hausdorff dimension of subgroups and quotients. The following lemma relates $\Hsthdim$ with the Hausdorff dimension on $H$ induced in the natural way by an $R$-standard filtration $\{S_n \}_{n \in \N}$ of $G$, i.e., $\hdim_{\{ H \cap S_n \}}^{H}.$ 

\begin{lemma}
\label{filtrazioak tekniko}
Let $G$ be a compact $R$-analytic group and $H$ an $R$-analytic subgroup of $G$. Then $\hdim_{\{H\cap S_n\}}^H(K) =  \Hsthdim(K)$ for all $K \leq_c H,$ where $\{S_n\}_{n \in \N}$ is an $R$-standard filtration of $G.$ 
\end{lemma}

\begin{proof}
Firstly, let $\{S_n\}_{n \in \N}$ and $\{T_n\}_{n \in \N}$ be two $R$-standard filtrations of $G.$ By Lemma \ref{KTZ5.3} and Theorem \ref{standard} it is straightforward that
\begin{equation}
\label{stfil}
\hdim_{\{ H \cap S_n\}}^H (K) = \hdim_{\{ H \cap T_n\}}^H (K),  \ \forall K \leq_c H.
\end{equation}

Secondly,  we shall show that there exists an open $R$-standard subgroup $S$ of $G$ such that $\{ H \cap S_n \}_{n \in \N}$ is an $R$-standard filtration of $H$. Then for any $R$-standard filtration $\{T_n \}_{n \in \N}$ of $G,$ by (\ref{stfil}) we have that 
$$\hdim^{H}_{\{ H \cap T_n \}} (K) = \hdim^{H}_{\left\{ H \cap S_n \right\}} (K) = \Hsthdim(K)$$
 for all $K \leq_c H,$ as desired. \\
 
 Let $d=\dim{G}$ and $k = \dim{H}$, since $H$ is an $R$-analytic subgroup there exists an $R$-chart $(U, \phi)$ of $1$ in $G$ such that 
\begin{align*}
\phi(H \cap U) & =   \left\{ (x_1, \dots, x_d) \in \left(\mideal^N \right)^{(d)} \middle| \ x_{k+1} = \dots = x_d = 0 \right\} \\ &
 = \left(\mideal^{N} \right)^{(k)} \times \{ 0 \}^{(d-k)},
\end{align*}
for some $N \geq 1,$ and $\phi(1)=\textbf{0}.$ Furthermore, since $U$ is open in $G$, from the proof of \cite[Theorem 13.20]{DDMS}  there exists an open $R$-standard subgroup $S$ of $G,$  of level $L \geq N,$ contained in $U$ and with homeomorphism $\phi|_S.$ Then
\begin{align*}
\phi\left(H \cap S \right)  & = 
\phi\left(S \right) \cap \phi(H \cap U) \\ & =
\left(\mideal^L\right)^{(d)} \cap \left(\left( \mideal^N \right)^{(k)} \times \{ 0\}^{(d-k)} \right) =  \left(\mideal^L \right)^{(k)} \times \{ 0\}^{(d-k)}.
\end{align*}
Therefore, if $\pi \colon (\mideal^{L})^{(k)} \times \{ 0 \}^{(d-k)} \rightarrow (\mideal^{L})^{(k)}$ is the natural homeomorphism, then $\left(H \cap S, \psi \right),$ where $\psi  = \pi \circ \phi|_{H \cap S},$ is an open $R$-standard subgroup of $H.$ 
Thus,  $$\psi\left(H \cap S_{n} \right) =  \pi\left(\phi(H \cap U) \cap \phi\left(S_{n} \right) \right) = \left(\mideal^{L + n} \right)^{(k)},$$
and one concludes that $\left\{ H \cap S_n \right\}_{n \in \N}$ is an $R$-standard filtration of $H.$ 
 \end{proof}
 
  We will focus on the case $R =\F_p[[t]]$ for quotients, since it is known (cf. \cite[Part II, Section IV.5, Remarks 2]{Serre}) that if $G$ is an $\F_p[[t]]$-analytic group and $N \unlhd G$ is a normal $\F_p[[t]]$-analytic subgroup, then $G/N$ is an $\F_p[[t]]$-analytic group. Hence, we shall relate the standard spectrum of the group and the spectrum of its analytic quotients. 

\begin{lemma}
\label{zatsta}
Let $G$ be a compact $\F_p[[t]]$-analytic group, $\{ S_n \}_{n \in \N}$ an $\F_p[[t]]$-standard filtration of $G$ and  $N \unlhd G$ a normal $\F_p[[t]]$-analytic subgroup of $G.$ Then
$$\sthdim(H) = \hdim_{\left\{ \frac{S_n N}{N} \right\}}(H),$$
for every $H \leq_c G/N.$
\end{lemma}
\begin{proof}
Let us fix some notation: let $R$ be the pro-$p$ domain  $\F_p[[t]]$ with maximal ideal $\mideal= (t),$ $d = \dim{G}$ and $e = \dim{G/N};$ let $\pi$ be the quotient map and let $\pr \colon \mideal^{(d)} \rightarrow \mideal^{(e)}$ be the projection onto the last $e$ coordinates.\\

Firstly, if $\{S_n\}_{n \in \N}$ and $\{T_n\}_{n \in \N}$ are two $R$-standard filtrations of $G,$ as in the proof of \cite[Theorem 3.1] {FGG} it can be seen that 
\begin{equation}
\label{stzat}
\hdim_{\left\{ \frac{S_n N}{N} \right\}} (H) = \hdim_{\left\{ \frac{T_nN}{N}\right\}} (H),  \ \forall H \leq_c G/N.
\end{equation}
Hence by (\ref{stzat}) it suffices to find an open  $R$-standard subgroup $S$ of $G$ such that $\{ S_n N/N \}_{n \in \N}$ is an $R$-standard filtration of $G/N.$ According to \cite[Part II, Section III.12]{Serre} there exists an $R$-chart $(U, \phi)$ of $1$ in $G$ adapted to $N,$ that is, $\phi(1) = \textbf{0}$ and 
$\pr \circ \phi(x) = \pr \circ \phi (y)$ if and only if $xy^{-1} \in N.$ Since $U$ is open in $G,$ from the proof of \cite[Theorem 13.20]{DDMS} there exists an open $R$-standard subgroup $S,$ of level $L,$ contained in $U$ with homeomorphism $\phi|_S.$ Let $\sigma \colon \pi(S) \rightarrow S$ be a continuous section such that $\sigma(1N) = 1$ (which exists by \cite[Proposition 2.2.2]{RZ}), then $\pi(S)$ is an $R$-standard subgroup of $G/N$, with level $L,$ dimension $e$ and homeomorphism $\psi = \pr \circ \phi \circ \sigma.$ Note that since $(U,\phi)$ is an adapted $R$-chart, the definition of $\psi$ is independent of the selected section and $\psi(S_nN/N) = \pr \circ \phi(S_n) = (\mideal^{L+n})^{(e)},$ so $\{ S_n N/ N\}_{n \in \N}$ is an $R$-standard filtration of $G/N.$ 
 \end{proof}

\section{Soluble compact $R$-analytic groups}

This section is devoted to proving Theorem \ref{T1}. 
\subsection{Abelian compact $R$-analytic groups}

Before dealing with soluble groups, we will prove the analogous result in the abelian case,  where $R$ is a general pro-$p$ domain of characteristic $p.$ We will use the following technical lemma (cf. \cite[Lemma 2.3]{FGG}).

\begin{lemma}
\label{azpikoa}
Let $(S, \phi)$ be an $R$-standard group of dimension $d$. Then there exists a non-constant polynomial $f$ such that $|S : S_{n}|  = p^{df(n)}$ for large enough $n$.
\end{lemma}

\begin{proposition}
\label{abelian} 
Let $R$ be a pro-$p$ domain of characteristic $p$ and let $(S, \phi)$ be an abelian $R$-standard group. Then $\sthspec(S) = [0,1].$
\end{proposition}

\begin{proof}
By \cite[Theorem 5.4]{KTZ} it suffices to prove that every finitely generated subgroup $H \leq_c S$ satisfies $\sthdim(H) = 0.$ Let $d$ be the dimension of $S$ and let $H \leq S$ be a topologically $r$-generated closed subgroup. Since the group operation in $S$ is given by a formal group law, by (\ref{fgl}) whenever $x \in S_n$ we have
$$\phi(x^p) \equiv  p \phi(x)  = \textbf{0} \mod{ \left(\mideal^{2n}\right)^{(d)}},$$
and thus $x^p \equiv 1 \pmod{S_{2n}}.$ Therefore $S_n/ S_{2n}$ is an elementary abelian $p$-group.\\

Since $S$ is abelian then $H/ (H \cap S_n)$ is an abelian $p$-group of exponent $p^e$ where $e \leq \lceil \log_2(n) \rceil.$ Moreover, since $H$ is topologically $r$-generated, it follows that $H/ (H \cap S_n)$ is $r$-generated and so
$$|H : H \cap S_n| \leq p^{er} \leq p^{ \lceil \log_2(n) \rceil r}.$$
Accordingly, by Lemma \ref{azpikoa}
$$\sthdim(H) = \liminf_{n \rightarrow \infty} \frac{\log_p{|H: H \cap S_n|}}{\log_p{|S: S_n|}}  \leq \liminf_{n \rightarrow \infty} \frac{r \lceil \log_2(n) \rceil }{df(n)} = 0,$$
as desired. 
 \end{proof}

Clearly, in view of of Corollary \ref{estandarrera}, this result can be generalised to compact abelian $R$-analytic groups.

\begin{corollary}
\label{abeldarra}
Let $R$ be a pro-$p$ domain of characteristic $p.$ If $G$ is an abelian compact $R$-analytic group, then $\sthspec(G) = [0,1].$
\end{corollary}
Furthermore, it is known that any $R$-standard group of dimension one is abelian (cf. \cite[Theorem 1.6.7]{Haz}), and we thus have the following:

\begin{corollary}
Let $R$ be a pro-$p$ domain of characteristic $p$ and let $G$ be a compact $R$-analytic group of dimension one. Then $\sthspec(G) = [0,1].$
\end{corollary}

\subsection{$\F_p[[t]]$-analytic subgroups}

Now we will mainly turn to the case when $R = \F_p[[t]].$ The main strategy to prove Theorem \ref{T1} lies in adding successive intervals to the spectrum, using the consecutive abelian quotients of a subnormal series. In fact, we have the following result. 

\begin{lemma}
\label{zatabel}
Let $G$ be a compact $\F_p[[t]]$-analytic group and let $N \unlhd K \leq G$ be $\F_p[[t]]$-analytic subgroups such that $\sthspec(K /N) = [0,1].$ Then 
$$ \left[\frac{\dim{N}}{\dim{G}} , \frac{\dim{K}}{\dim{G}} \right]   = \left[\sthdim(N) , \sthdim(K) \right] \subseteq \sthspec(G).$$
\end{lemma}

\begin{proof}
By \cite [Main Theorem]{FGG} one has $\sthdim{(H)} = \dim{H}/{\dim{G}} $ for any $\F_p[[t]]$-analytic subgroup $H$
and a closer scrutiny of its proof reveals that such dimension is given by a proper limit; so the result is straightforward from Corollary \ref{korzatidura1}, Lemma \ref{filtrazioak tekniko} and Lemma \ref{zatsta}. 
 \end{proof}
  
Thus, we shall establish a useful criterion for finding $\F_p[[t]]$-analytic subgroups of a compact $\F_p[[t]]$-analytic group. The main obstacle compared with classical Lie theory arises here: it is well known that any closed subgroup of a  real ($p$-adic) Lie group is a  real ($p$-adic) Lie subgroup; nevertheless for $R$-analytic groups, closedness is a necessary condition, but not sufficient. For example, the additive group $\F_p[[t]]$ is an $\F_p[[t]]$-standard group and $\F_p\!\left[\!\left[t^2 \right]\!\right]$ is a closed subgroup with its own $\F_p[[t]]$-standard group structure. However, those manifold structures are not compatible and $\F_p\!\left[\!\left[t^2 \right]\!\right]$ is not an $\F_p[[t]]$-analytic subgroup of $\F_p[[t]].$

\begin{remark}
Denote by $ \F_p(\!(t)\!)$ the local field of characteristic $p$ and valuation ring $\F_p[[t]]$, and let $M$ be an $\F_p[[t]]$-analytic manifold. Since $\F_p[[t]]$ is a discrete valuation ring, $M$ has also an $\F_p(\!(t)\!)$-analytic manifold structure  (cf. \cite[Section 13.1]{DDMS}). 
\end{remark}
 
The task of finding $\F_p(\!(t)\!)$-analytic subgroups will be carried out by using a generalisation from \cite{JaKl} which shows that homogeneous subsets have a manifold structure over the local field $\F_p(\!(t)\!).$ According to the  definition therein (cf. \cite[Section 4]{JaKl}) a set $X \subseteq M$ is an \emph{analytic subset} if for each $x \in X$ there exists an open neighbourhood $U$ of $x$ and some $\F_{p}(\!(t)\!)$-analytic functions $f_1, \dots, f_r$ defined on $U$ (for some $r  = r_x$) such that 
$$X \cap U = \{ y \in U \mid f_i (y) = 0 \ \forall i =1, \dots, r \}.$$
In other words, an analytic subset is locally the nullset of some analytic functions. We then have (cf. \cite[Corollary 4.2]{JaKl}):

\begin{theorem}
\label{subset}
Let $G$ be an $\F_p[[t]]$-analytic group and let $H$ be both a subgroup of $G$ and an analytic subset of $G.$ Then $H$ is an $\F_p[[t]]$-analytic subgroup of $G.$
\end{theorem}
Let us see some examples of applications of the preceding lemma:

\begin{corollary}
\label{centre+centraliser}
Let $G$ be an $\F_p[[t]]$-standard group and  $a$ in $G.$ Then $Z(G)$ and $C_G(a)$ are $\F_p[[t]]$-analytic subgroups.
\end{corollary}

\begin{proof}
By the previous theorem it is enough to show that $Z(G)$ and $C_G(a)$ are analytic subsets. The former is proved in \cite[Corollary 4.3]{JaKl}, while the latter follows the same spirit. 
Indeed, since $G$ is $\F_p[[t]]$-standard of level say $N$ and dimension say $d,$ then it can be identified with $\left(t^N \right)^{(d)}.$ Hence, for every $x \in G,$ one has that $\F_p[[t]][[X_1, \dots, X_d ]]$ is a subring of the local ring of functions at $x.$ For $(\alpha_1, \dots, \alpha_d) \in \N^{(d)}$ define $|\alpha| = \sum_{i=1}^{d} \alpha_i.$ Since the group operation is given by a formal group law, by (\ref{fgl}) there exist some $g_{i,\alpha} \in \F_p[[t]][[X_1, \dots, X_d]]$ such that 
$$\pi_i\left(y^{-1} a y \right) =  a_i + \sum_{|\alpha| \geq 1} g_{i, \alpha}(a) y_1^{\alpha_1} \dots y_d^{\alpha_d} = a_i + h_i(y),$$
for all $y$ in $G,$ where the map $\pi_i \colon \left(t^N \right)^{(d)} \rightarrow \left(t^N \right)$ is the projection to the $i$th coordinate. Then the maps $h_i(y) = \sum_{|\alpha| \geq 1} g_{i, \alpha}(a) y_1^{\alpha_1} \dots y_d^{\alpha_d} $ are clearly $\F_p[[t]]$-analytic. Therefore 
\begin{align*}
C_G(a) &= 
\left\{ y \in G \mid \pi_i\left(y^{-1} a y\right) = a_i  \ \forall i=1, \dots, d \right\} \\& =
\left\{ y \in G \mid  h_i(y) = 0 \  \forall i = 1, \dots, d\right\},
\end{align*}
and  $C_G(a)$ is an analytic subset. 
 \end{proof}

\begin{corollary}
\label{algebraic}
Let $G \subseteq \GL_n(\F_p[[t]])$ be a linear $\F_p[[t]]$-analytic group and let $\Hcal$ be a Zariski closed subgroup of  $\GL_n(\F_p[[t]]).$ Then $\Hcal  \cap G$ is an $\F_p[[t]]$-analytic subgroup of $G.$ 
\end{corollary}

\begin{proof}
 Since $\mathcal{H}$ is closed in the Zariski topology it is an affine set, that is, there exists a subset $I$ of $\F_p[[t]][\mathbf{X}],$ where $\mathbf{X}$ is a tuple of $n^2$ variables, such that
$$\mathcal{H}  =\left \{ y \in \GL_n(\F_p[[t]]) \mid f_{i}(y) = 0 \ \forall f_i \in I \right\}.$$
 But since $\F_p[[t]][\mathbf{X}]$ is Noetherian we can assume $I$ to be finite, and thus
$$\mathcal{H} \cap G = \left\{ y \in G \mid f_i(y) =0 \ \forall f_i \in I \right \}$$
is an analytic subset, so it is an $\F_p[[t]]$-analytic subgroup by Theorem \ref{subset}. 
 \end{proof}

We are now in position to prove the main theorem by using the previous results:

\begin{proof}[Proof of Theorem \ref{T1}]
By Corollary \ref{estandarrera} we can assume without loss of generality that $G$ is $\F_p[[t]]$-standard. We first prove the theorem for the case when $G$ is linear over $\F_p[[t]],$ that is, $G \subseteq \GL_n(\F_p[[t]]).$  Let  $\mathcal{G}$ be the Zariski closure of $G$ in $\GL_n(\F_p[[t]]).$ According to \cite[Theorem 5.11]{Weh} $\mathcal{G}$ is a soluble algebraic group, so there exists a soluble series 
$$\mathcal{G} = \mathcal{H}_1 \unrhd \mathcal{H}_2 \unrhd \dots \unrhd \mathcal{H}_{k-1} \unrhd \mathcal{H}_{k} = \{ 1 \}$$
of Zariski closed subgroups. Then 
$$G = \mathcal{H}_1 \cap G \unrhd \mathcal{H}_2 \cap G \unrhd \dots \unrhd \mathcal{H}_{k-1} \cap G \unrhd \mathcal{H}_{k}  \cap G= \{ 1 \}$$
is a soluble series of $G$ given by $\F_p[[t]]$-analytic subgroups, by Corollary \ref{algebraic}. \\

Denote $H_i = \mathcal{H}_i \cap G.$ Since each $H_i$ is an $\F_p[[t]]$-analytic subgroup of $G$ then $H_{i-1} / H_{i}$  is a compact abelian $\F_p[[t]]$-analytic group for all $i \in \{2, \dots, k \}$, so by Corollary \ref{abeldarra} it follows that $\sthspec(H_i / H_{i-1}) = [0,1].$ Hence by Lemma \ref{zatabel}  one has $[\sthdim(H_{i}), \sthdim(H_{i-1})] \subseteq \sthspec(G)$ for all $i \in \{2, \dots, k\}$ and thus $\sthspec(G) = [0,1].$\\
Let us finally turn to the general case. By Corollary \ref{centre+centraliser}, $Z(G)$ is an abelian $\F_p[[t]]$-analytic subgroup of $G$ and thus by Corollary \ref{abeldarra} and Lemma \ref{zatabel} 
$$[0, \sthdim{Z(G)}] \subseteq \sthspec(G).$$ 
Moreover, by  \cite[Proposition 5.1] {Jai}  one has that $G/Z(G)$ is a compact  soluble linear $\F_p[[t]]$-analytic group. Hence, according to Lemma \ref{filtrazioak tekniko} and Lemma \ref{zatsta}, 
$$ \hspec_{\left\{S_n Z(G) / Z(G)\right\}}(G/Z(G)) = \sthspec{(G/Z(G))}  = [0, 1],$$ 
and so by Corollary \ref{korzatidura1}
$$\left[\sthdim{Z(G)}, 1 \right]   \subseteq \sthspec(G),$$
thus obtaining the whole interval in the spectrum. 
\end{proof}

\section{Compact $\F_p[[t]]$-analytic groups}
In this section, we first prove Theorem \ref{T2} and subsequently we study the Hausdorff spectrum of some classical Chevalley groups. 
\subsection{Proof of Theorem \ref{T2}}

The previous section suggests that a suitable way {to find} an interval in the $\F_p[[t]]$-standard Hausdorff spectrum of a compact $\F_p[[t]]$-analytic group $G$ is looking for a soluble $\F_p[[t]]$-analytic subgroup. This search will rely heavily on the topological analogue of the Tits alternative. But we first observe the following: 
\begin{lemma}
\label{interval1}
Let $G$ be an $\F_p[[t]]$-standard group. Suppose that either
\begin{itemize}
\item[\textup{(i)}] $Z(G)$ is infinite or 
\item[\textup{(ii)}] $G$ contains an element $x$ of infinite order.
\end{itemize}
 Then $\left[0, \sfrac{1}{\dim{G}} \right] \subseteq \hspec(G).$
\end{lemma}

\begin{proof}
Under the first hypothesis, by Corollary \ref{centre+centraliser} $Z(G)$ is an abelian infinite $\F_p[[t]]$-analytic subgroup. Similarly, under the second hypothesis $Z(C_G(x))$ is an abelian $\F_p[[t]]$-analytic subgroup which is infinite, because $\langle x \rangle \leq Z(C_G(x)).$ In both cases, since $G$ is compact, there exists an abelian $\F_p[[t]]$-analytic subgroup of positive dimension whose $\F_p[[t]]$-standard spectrum is the whole interval $[0,1],$ according to Proposition \ref{abelian}, thus the result follows by Lemma \ref{zatabel}. 
 \end{proof}

\begin{proof}[Proof of Theorem \ref{T2}]
First, observe that when $Z(G)$ is infinite the result follows by Lemma \ref{interval1}(i), so we shall deal with the case when $Z(G)$ is finite. But then $G/Z(G)$ is an $\F_p[[t]]$-analytic group of dimension $\dim{G}$ and according to Corollary \ref{zatidura2} it follows that
$$\sthspec(G) = \sthspec(G/Z(G)).$$
Furthermore, using \cite[Proposition 5.1]{Jai} we have that $G/ Z(G)$ is an $\F_p(\!(t)\!)$-analytic group which is linear over $\F_p(\!(t)\!).$ Hence by the topological Tits alternative (cf. \cite[Theorem 1.3]{BG}) it follows that $G/Z(G)$ contains either an open soluble subgroup, say $H,$ or contains a dense free subgroup. In the former case, $H$ is a soluble $\F_p[[t]]$-analytic group of dimension $\dim{G/ Z(G)} = \dim{G}$ and thus 
$$ \sthspec(G/ Z(G)) =  \left[ 0, 1 \right].$$
In the latter case $G/Z(G)$ contains an element of infinite order and the statement follows from Lemma \ref{interval1}(ii). 
\end{proof}

\subsection{Classical Chevalley groups}

Even though the previous result ensures the existence of a real interval of type $[0, \alpha]$ in the standard Hausdorff spectrum,  there is no general method to find the maximum value of $\alpha.$ 

However, when $G \subseteq \GL_n(\F_p[[t]])$ is linear over $\F_p[[t]]$ we can use the theory of algebraic groups. Indeed, the \emph{Borel subgroup} of $\GL_n(\F_p[[t]])$ -- i.e. a maximal connected soluble algebraic subgroup of $\GL_n(\F_p[[t]]),$ which is unique up to conjugation in $\GL_n(\F_p[[t]])$ (cf. \cite[Theorem 21.3]{Hump}) -- is the set of invertible $n \times n$ upper triangular matrices, say $\B = T_n(\F_p[[t]]).$ Since $\B$ is an algebraic group, by Corollary \ref{algebraic} it follows that $\B(G) := \B \cap G$ is a soluble $\F_p[[t]]$-analytic subgroup of $G.$  In particular, we can use this fact in order to describe the Hausdorff spectrum of the classical Chevalley groups with coefficients in the ring $\F_p[[t]].$ 

The Chevalley group over $\F_p[[t]]$ associated to a root system of type $A_n$ ($n \geq 1$)  is $\SL_{n+1}(\F_p[[t]]),$ referred to as the \emph{special linear} group. It is well known that $\SL_n(\F_p[[t]])$ is a compact $\F_p[[t]]$-analytic group of dimension $n^2-1$, containing as an open subgroup the $\F_p[[t]]$-standard group 
$$\SL_n^1 (\F_p[[t]]):= \ker\bigl\{ \SL_n(\F_p[[t]])  \longrightarrow \SL_n\left(\F_p[[t]]/ t \F_p[[t]] \right)\bigr\}$$ 
(cf. \cite[Exercise 13.9]{DDMS}).  For this first classical group we recover the following description of its $\F_p[[t]]$-standard spectrum, already proved in \cite{BSh} for the congruence subgroup filtration of $\SL_n(\F_p[[t]])$ --- which is indeed an $\F_p[[t]]$-standard filtration of $\SL_n(\F_p[[t]]).$

\begin{corollary}\textup(\textup{cf. \cite[Proposition 4.4]{BSh}}\textup)
\label{SL}
The $\F_p[[t]]$-standard Hausdorff spectrum of $\SL_n(\F_p[[t]])$ contains the real interval $\left[0, \frac{n(n+1) -2}{2n^ 2 -2}\right].$
\end{corollary}
\begin{proof}  
Note that $\B(\SL_n(\F_p[[t]])) = \B \cap \SL_n(\F_p[[t]])$ is the soluble subgroup of upper triangular matrices with determinant $1$ and entries in $\F_p[[t]],$ which is an $\F_p[[t]]$-analytic subgroup of dimension $\frac{n(n+1)}{2}-1.$ The result follows by Theorem \ref{T1} and Lemma \ref{zatabel}. 
 \end{proof}

 This method can be also used with the remaining classical Chevalley groups over a general pro-$p$ domain $R$.
\begin{itemize}
\item A root system of type $B_n$ $(n \geq 2)$ defines the odd \emph{special orthogonal  group}  
$$\SO_{2n +1}(R) := \left\{ A \in \Mat_{2n+1}(R) \ \middle| \ A^t K_{2n+1} A = K_{2n+1} \right\},$$
where $K_n = \begin{pmatrix}
0 & \dots & 0 & 1\\
0& \dots & 1 & 0 \\
\vdots  & \reflectbox{$\ddots$}  & \vdots & \vdots  \\
1 & \dots & 0 & 0 
\end{pmatrix} \in \Mat_n(R),$ which is an $R$-analytic group of dimension $n(2n+1).$
\item A root system of type $C_n $ $(n \geq 3)$ defines the \emph{symplectic group} 
$$\Sp_{2n}(R) := \left\{ A \in \Mat_{2n}(R) \mid A^t J_{2n} A = J_{2n} \right\},$$
where $J_{2n} = \begin{pmatrix}
0 & K_n \\
- K_n & 0 
\end{pmatrix},$  which is an $R$-analytic group of dimension $n(2n+1).$

\item A root system of type $D_n $ ($n \geq 4$) defines the even \emph{special orthogonal group}
$$\SO_{2n}(R) := \left\{ A \in \Mat_{2n}(R) \mid A^t K_{2n} A = K_{2n} \right\},$$
which is $R$-analytic of dimension $n(2n-1).$
\end{itemize}

Again, all these groups contain an $R$-standard subgroup of the same dimension (cf. \cite[Exercise 13.11]{DDMS}), say $S.$ Then $S$ is an open $R$-standard subgroup, and since according to the following lemma they are compact $R$-analytic groups it follows that they are in fact profinite groups.  

\begin{lemma}
Let $R$ be a pro-$p$ domain. Then $\SO_n(R)$ and $\Sp_{2n}(R)$ are compact topological spaces. 
\end{lemma}
\begin{proof}
Since $R / \mathfrak{m}$ is finite then $\widehat{R},$ the completion of $R$ with respect to the $\mathfrak{m}$-adic topology, is compact. But $R$ is complete so $R = \widehat{R}$ is compact, and thus $\Mat_n(R) = R^{(n^2)}$ is compact. Hence, the closed subgroups $\SO_n(R)$ and $\Sp_{2n}(R)$ are compact. 
 \end{proof}
We  are now in a position for describing the Hausdorff spectrum of those profinite groups. 
\begin{corollary}
\label{alphades}
For any $n \geq 1$ 
\begin{enumerate}
\item[\textup{(i)}] $\sthspec \left(\Sp_{2n}(\F_p[[t]]) \right)$ contains the real interval $\left[ 0, \frac{n+1}{2n+1} \right],$
\item[\textup{(ii)}] $\sthspec\left(\SO_{2n}(\F_p[[t]])\right)$ contains the real interval $\left[0, \frac{n}{2n-1} \right]$ and
\item[\textup{(iii)}] $\sthspec\left(\SO_{2n+1}(\F_p[[t]])\right)$ contains the real interval $\left[0, \frac{n+1}{2n+1} \right].$
\end{enumerate}
\end{corollary}
\begin{proof}
Denote  $R = \F_p[[t]].$ \\
\linebreak
(i) Note that $\B\left(\Sp_{2n}(R)\right) = \Sp_{2n}(R) \cap T_{2n}(R),$
which after a simple computation (cf. \cite[Example 6.7(4)]{MalTes}) can be seen to coincide with the product of the set 
$$\left\{  \begin{pmatrix}
A & 0 \\
0 & K_n A^{-t} K_n  
\end{pmatrix} \ \middle| \  A \in T_n(R) \right \}$$
 with $\left\{  \begin{pmatrix}
I_n & K_nS \\
0 & I_n 
\end{pmatrix} \ \middle| \  S \in \Mat_n(R) \textrm{ is symmetric} \right \}$.\\

 Hence, $\B(\Sp_{2n}(R))$ is a soluble $\F_p[[t]]$-analytic subgroup of dimension $n^2 +n,$ and thus by Theorem \ref{T1} and Lemma \ref{zatabel} we have
$$ \left[ 0, \frac{n+1}{ 2n+1}\right] =  \left[ 0, \frac{\dim{\B(\Sp_{2n} (R))}}{ \dim{\Sp_{2n}(R)}}\right] \subseteq \sthspec(\Sp_{2n}(R)).$$
(ii) In much the same way one has $\B(\SO_{2n}(R)) = \SO_{2n}(R) \cap T_{2n}(R),$ which is a  soluble $\F_p[[t]]$-analytic subgroup of dimension $n^2,$ so
$$ \left[ 0, \frac{n}{2n-1}\right] =  \left[ 0, \frac{\dim{\B(\SO_{2n} (R))}}{ \dim{\SO_{2n}(R)}}\right] \subseteq \sthspec(\SO_{2n}(R)).$$
(iii) Similarly one has  $\B(\SO_{2n +1}(R)) = \SO_{2n+1}(R) \cap T_{2n+1}(R),$ which is a  soluble $\F_p[[t]]$-analytic subgroup of dimension $n^2 +n,$ so
$$ \left[ 0, \frac{n+1}{ 2n+1}\right] =  \left[ 0, \frac{\dim{\B(\SO_{2n+1} (R))}}{ \dim{\SO_{2n+1}(R)}}\right] \subseteq \sthspec(\SO_{2n+1}(R)).  $$

 \end{proof}

Note in passing that for any classical Chevalley group one has $\alpha \geq 1/2.$ \\

Finally, we shall provide examples of compact $\F_p[[t]]$-analytic groups whose spectrum is not the whole interval. More precisely,  we will show that in most of the classical Chevalley groups $1$ is an isolated point in the spectrum, thus proving that $\alpha < 1.$ 

In passing, we note that in \cite[Theorem 1.4]{BSh} it is proved that if  $p > 2$ then
$$\sthspec\left(\SL_n^1(\F_p[[t]]\right)  \cap \left(1 - \frac{1}{n+1}, 1  \right) = \varnothing.$$

We will prove an analogous result for other classical Chevalley groups following the same technique and working in the corresponding graded Lie algebra. Given an $R$-analytic group and a $p$-central series $\{G_n\}_{n \in \N}$ (note that by \cite[Proposition 13.22]{DDMS} any $R$-standard filtration is a $p$-central series),  we can define the restricted graded Lie $\F_p$-algebra $\mathcal{L}(G) = \bigoplus_{n \geq 0} G_{n} /G_{n+1}$ (cf. \cite[Definition 2.9]{LuSh}). Any closed subgroup  $H \leq G$ defines a graded subalgebra of $\mathcal{L}(G)$,  which by abuse of notation we will denote by $\mathcal{L}(H),$ and is given by 
$$\mathcal{L}(H) = \bigoplus_{n \geq 0} \frac{(H \cap G_n)G_{n+1}}{ G_{n+1}}.$$
Although every closed subgroup defines a graded subalgebra, there might be graded subalgebras that do not arise in this way. \\
\linebreak
\textit{Notation.}
Since $\dim$ usually denotes the analytic dimension of a manifold, henceforth $\dim_F$ will be used to denote the $F$-vector space dimension. \\

Given a graded $\F_p$-algebra $L = \oplus_{n \geq 0} L_n$ and a graded $\F_p$-subalgebra $K = \bigoplus_{n \geq 0} K_{n},$  the \emph{Hausdorff density} is defined as follows
$$\hD(K) := \liminf_{n \rightarrow \infty} \frac{\sum_{m \leq n} \dim_{\F_p}{K_m}}{\sum_{m \leq n} \dim_{\F_p}{L_m}}.$$

Clearly, in view of the preceding definitions for any closed subgroup $H$ we have $\hD(\mathcal{L}(H)) = \hdim_{\{ G_n \}}(H)$ (cf. \cite[Lemma 5.1]{BSh}).\\

Let now $F$ be a field and  $\G$  a finite dimensional perfect (i.e. $[\mathcal{G}, \mathcal{G}]= \mathcal{G}$) $F$-algebra; then we can consider the infinite dimensional $F$-algebra $\mathcal{G} \otimes_F tF[t]$ with Lie bracket defined by $[A \otimes t^n, B \otimes t^m ] := [A, B]_{\mathcal{G}} \otimes t^{n +m}$ on elementary tensors. We now note the following:

\begin{lemma}
\label{Zorn}
Let $\mathcal{L} = \mathcal{G} \otimes_F tF[t]$ be as above. Then, any graded $F$-subalgebra of infinite codimension is contained in a graded $F$-subalgebra of infinite codimension, maximal with respect to that property. 
\end{lemma}

\begin{proof}

Firstly,  $\mathcal{L}$ is a finitely generated $F$-algebra. Indeed, let $\{ x_1, \dots, x_m \}$ be a generating set of $\G,$ then $S = \{ x_1, \dots, x_m, x_1 \otimes t, \dots, x_m \otimes t \} $ generates $\mathcal{L}.$ Indeed, $\langle S \rangle_F$ contains $\mathcal{G}$ and $\mathcal{G} \otimes t;$ and assume by induction that $\langle S \rangle_F$ contains $\G \otimes t^{n-1}.$ Then, since $\G$ is perfect
$$\mathcal{G} \otimes t^n = \left[\mathcal{G}, \mathcal{G} \right] \otimes t^n = \left[\mathcal{G} \otimes t^{n-1}, \mathcal{G} \otimes t \right] \subseteq \langle S \rangle_F.$$

Now, the result follows by Zorn's Lemma. Indeed, consider the set of graded $F$-subalgebras of infinite codimension, which is partially ordered under inclusion. Let $\{ H_i\}_{i \in I}$ be a totally ordered subset of graded $F$-algebras of infinite codimension and consider $H = \cup_{i \in I} H_i,$ which is a graded $F$-subalgebra of $\mathcal{L}.$ Suppose by contradiction that $H$ has finite codimension in $\mathcal{L}$, and so it is a finitely generated $F$-algebra. Assume that $H = \langle h_1, \dots, h_r \rangle_F,$ then there exists an $i_0 \in I$ such that $h_k \in H_{i_0}$ for all $k \in \{ 1, \dots, r \}$ and so $H = H_{i_0}$ has infinite codimension in $\mathcal{L},$ which is a contradiction. Hence  $\{ H_i\}_{i \in I}$ has a maximal member with respect to inclusion, which concludes the proof. 
 \end{proof}

If one requires central simplicity, we have the following result (cf. \cite[Corollary 5.3]{BSh}) bounding the Hausdorff density of graded subalgebras that are maximal with respect to having infinite codimension.

\begin{theorem}
\label{BSh}
Let $\mathcal{G}$ be a central simple algebra over a field $F$ and let $\mathcal{L} = \mathcal{G} \otimes_{F} t F[t].$ Then the density of a graded subalgebra that is maximal with respect to having infinite codimension is either $1/q,$ where $q$ is a prime, or $\dim_F{\mathcal{H}}/ \dim_F{\mathcal{G}},$ where $\mathcal{H}$ is a maximal graded subalgebra of $\mathcal{G}.$
\end{theorem}

\begin{remark}
\label{simple}
Recall that a finite dimensional algebra over a field $F$ is called central simple when it is simple and its centroid coincides with $F.$ Nevertheless, if $F$ is finite, the assumption of the previous theorem can be weakened to only requiring that $\mathcal{G}$ is simple. Indeed, the previous theorem is a corollary of \cite[Theorem 4.1]{BSZ} and it is pointed out in \cite[Remark after Theorem 1.1]{BSZ} that when $F$ is finite simplicity of  $\mathcal{G}$ is enough.
\end{remark}

Now, we apply this result to see that $1$ is an isolated point in the standard spectrum of most of the classical Chevalley groups. 

\begin{corollary}
\label{isolatua}
Let $X_{n}$ be a root system of type $A_n$ $(n \geq1)$, $B_n$ $(n \geq 2)$,  $C_n $ $(n \geq 3)$ or $D_{n}$ $(n \geq 4),$ let $G = G(X_n)$ be the classical Chevalley group associated to $X_n$ on $ \F_p[[t]]$ and  $L(Q)$  the algebra associated to that root system on an arbitrary ring $Q.$ If $L(\F_p)$ is simple, then
$$\sthspec(G) \cap \left(1- \frac{1}{\dim{G}},1 \right) = \varnothing.$$ 
\end{corollary}

\begin{proof}
On the one hand, 
by \cite[Exercise 13.11]{DDMS} it follows that $G$ contains an open $\F_{p}[[t]]$-standard subgroup, say $S,$ such that $\mathcal{L}(S) \cong L(\F_p[[t]]).$ Furthermore, by \cite[Proposition 13.27]{DDMS} there is an isomorphism $\mathcal{L}(S) \cong L_0 \otimes_{\F_p} \grm{\mideal}$ as $\F_p$-vector spaces where 
$$L_0 = L(\F_p[[t]]) / t L(\F_p[[t]]) \cong L(\F_p)\quad  \textrm{ and } \quad \grm{\mideal} = \bigoplus_{n \geq 1} \left(t^n \right) / \left(t^{n+1}\right).$$
Hence $\mathcal{L}(S) \cong L(\F_p) \otimes_{\F_p} t\F_p[t].$\\

On the other hand, let $H \leq_c S$ be a closed subgroup such that $\sthdim(H) < 1.$ Then $|S:H|$ is infinite and so $\mathcal{L}(H)$ has infinite codimension in $\mathcal{L}(S).$ Since $L(\F_p)$ is simple, according to Lemma \ref{Zorn} we have that $\mathcal{L}(H)$ is contained in a graded subalgebra of $\mathcal{L}(S),$ say $\mathcal{M},$ maximal with respect to having infinite codimension. Hence by Theorem \ref{BSh} and Remark \ref{simple} we have
\begin{align*}
\sthdim(H)  &= 
 \hD(\mathcal{L}(H)) \leq \hD(\mathcal{M}) \\& \leq
\max\left\{ \frac{1}{2}, \frac{\dim_{\F_p}{\mathcal{H}}}{ \dim_{\F_p}{L(\F_p)} } \ \middle| \ \mathcal{H} \textrm{ maximal subalgebra of } L(\F_p) \right\} \\& \leq
1 - \frac{1}{\dim_{\F_p}{L(\F_p)}} = 1 - \frac{1}{\dim{S}},
\end{align*}
because $\dim{S} = \dim{G(X_n)} = \dim_{\F_p}L(\F_p).$ Therefore, the result follows since $\sthspec(G) = \sthspec(S).$ 
 \end{proof}

Finally, the classical Chevalley algebras $\mathfrak{so}_{n}(F)$ and $\mathfrak{sp}_{2n}(F)$ over a field of positive characteristic $p$ have been thoroughly studied. When $p =2$ none of them is simple, but when $p \geq 3$ it is well known that $\mathfrak{so}_{n}(F)$ $(n \geq 5)$ and $\mathfrak{sp}_{2n}(F)$ $(n \geq 2)$ are simple algebras (cf. \cite{Str}). Hence we deduce that: 

\begin{corollary}
\label{batiso}
Let $p \geq 3$ be a prime and assume $G$   is either $\SO_{n}(\F_p[[t]])$ $(n \geq 5)$ or $\Sp_{2n}(\F_p[[t]])$ $(n \geq 2)$.
 Then  
$$\sthspec(G) \cap \left(1- \frac{1}{\dim{G}},1 \right) = \varnothing.$$ 
\end{corollary}

Classical Chevalley groups over the local field $\F_p(\!(t)\!)$ are linear simple algebraic groups. More generally, if $G$ is an algebraic group over the local field $F_p(\!(t)\!),$ then the group of $\F_p[[t]]$-rational points, $G(\F_p[[t]]),$ admits naturally an $\F_p[[t]]$-analytic manifold structure (cf. \cite[Proposition I.2.5.2]{Mar}). Hence, the above result suggests the following conjecture:

\begin{conjecture}
Let $G$ be a linear $\F_p(\!(t)\!)$-algebraic semisimple group. Then
$$\sthspec(G(\F_p[[t]])) \cap \left(1 - \frac{1}{\dim{G(\F_p[[t]])}} , 1 \right) = \varnothing.$$
\end{conjecture}

\section*{Acknowledgements}
The authors would like to thank G.A. Fern\'andez-Alcober for his helpful comments.

\end{document}